\documentclass[12pt]{amsart}

\usepackage{amssymb,amsmath,graphics,verbatim,bm}
\usepackage{latexsym}
\usepackage{eucal}
\usepackage{a4wide}

\newtheorem{theorem}{Theorem}
\newtheorem{lemma}[theorem]{Lemma}
\newtheorem{proposition}[theorem]{Proposition}
\newtheorem{corollary}[theorem]{Corollary}
\theoremstyle{definition}
\newtheorem{definition}[theorem]{Definition}
\theoremstyle{remark}

\newcommand{\ad}{\mathrm{ad}}
\newcommand{\bC}{\mathbb{C}}
\newcommand{\bR}{\mathbb{R}}
\newcommand{\bZ}{\mathbb{Z}}
\newcommand{\bH}{\mathbb{H}}
\newcommand{\cs}{\mathfrak{cs}}
\newcommand{\Cot}{\mathfrak{Cott}}
\newcommand{\cott}{\mathfrak{cott}}
\newcommand{\CS}{\mathfrak{CS}}

\newcommand{\cP}{\mathcal{P}}
\newcommand{\cV}{\mathcal{V}}
\newcommand{\cZ}{\mathcal{Z}}
\newcommand{\dCS}{\overset{\bm{.}}{\CS}}
\newcommand{\dg}{\dot{g}}
\newcommand{\gl}{\mathrm{gl}}
\newcommand{\II}{\mathrm{I\! I}}
\newcommand{\mfa}{\mathfrak{a}}
\newcommand{\omc}{\omega_{\mathrm{MC}}}
\newcommand{\pt}{\partial_t}
\newcommand{\Psldc}{\mathrm{PSL}(2,\bC)}
\newcommand{\Ric}{\mathfrak{Ric}}
\newcommand{\scal}{\mathfrak{scal}}
\newcommand{\Sch}{\mathfrak{Sch}}
\newcommand{\sign}{\mathrm{sign}}
\newcommand{\signature}{\mathrm{signature}}
\newcommand{\Sot}{\mathrm{SO}(3)}
\newcommand{\Sud}{\mathrm{SU}(2)}
\newcommand{\sot}{\mathrm{so}(3)}
\newcommand{\sym}{\mathrm{sym}}
\newcommand{\tr}{\mathrm{tr}}
\newcommand{\vol}{\mathrm{vol}}

\begin{document}
\title[Cotton tensor and Chern-Simons invariants]{The Cotton tensor and Chern-Simons invariants 
in dimension $3$: an introduction}
\author{Sergiu Moroianu}
\thanks{Partially supported by the CNCS project PN-II-RU-TE-2012-3-0492.}
\address{Institutul de Matematic\u{a} al Academiei Rom\^{a}ne\\
P.O. Box 1-764\\RO-014700
Bu\-cha\-rest, Romania}
\email{moroianu@alum.mit.edu}
\date{\today}

\maketitle

\section{Motivation}

Let $(M,g)$ be an oriented Riemannian $3$-manifold. It is natural to ask if, 
like it happens in dimension $2$, the metric $g$ is locally conformally flat.
There exists an obstruction to local conformal flatness in dimension $3$, discovered by
\'Emile Cotton \cite{cotton} in 1899. This obstruction is a symmetric, traceless,
conformally covariant, and divergence-free $2$-tensor on $M$. In 1974
it was further unveiled by Chern and Simons \cite{cs} that the Cotton tensor is 
\emph{variational}: if $M$ is compact, there exists a $\bR/\bZ$-valued function 
on the space of Riemannian metrics, the Chern-Simons invariant,
whose gradient at a given metric $g$ is the Cotton tensor $\Cot(g)$.

In this survey paper we give a short and self-contained introduction to
the theory of Chern-Simons invariants for three-dimensional Riemannian manifolds.
As consequences, we will prove in a conceptual way some of the important
properties of $\Cot$.
Our approach relies on the paper by Chern and Simons \cite{cs},
with modern notation and focusing on the dimension $3$.
We use Besse's book \cite{besse} as reference for standard formulas in Riemannian geometry. 

Chern-Simons invariants are defined here with respect to some Riemannian metric, 
but they actually depend on just a connection. Moreover they can be defined 
in arbitrary dimensions, while we have limited ourselves
to dimension $3$. The main focus of interest regarding Chern-Simons invariants 
shifted in recent years towards invariants of smooth
manifolds, independent of any background metric, obtained by averaging Chern-Simons
invariants with respect to a (mathematically non-rigorous)
measure over an infinite-dimensional space of connections. Such an invariant is called a 
``topological quantum field theory'' and currently plays  a prominent role in mathematical physics. 
Witten's foundational paper \cite{w} highlighted a link between this Chern-Simons TQFT 
and the Jones polynomial of knots. We do not attempt here any forray into these fancy fields.

Since on one hand a comprehensive list of references would dwarf
the size of the paper, and since on the other hand our text 
is self-contained except for the last section, we have kept the references to a minimum, 
hoping that the numerous authors having significant contributions
to the subject will not feel slighted. We have limited our 
ambitions to understanding Riemannian Chern-Simons invariants of $3$-manifolds, 
in the goal of offering the reader a firm, albeit tiny, first foothold into Chern-Simons realm. 
The Cotton tensor comes as a reward for our self-limitation. 
Evidently, no claim of originality is made below other
than the presentation, itself strongly influenced by \cite{CS}. 

These notes are accessible to readers familiar with the basic objects of Riemannian 
geometry: differential forms,
Levi-Civita connection, curvature, Ricci tensor, scalar curvature.
They formed the topic of a concentrated doctoral course at the University 
of Bucharest in March 2014. I am grateful to Sorin D\u asc\u alescu for the invitation.

\subsection*{Acknowledgments}
I have learned about Chern-Simons invariants together 
with Colin Guillarmou, during our joint project \cite{CS} on the Chern-Simons bundle 
over the Teichm\"uller space. I thankfully acknowledge Colin's influence 
over my understanding of the subject.

\section{Definitions and background}

Throughout the paper, $M$ is a Riemannian manifold of dimension $3$, and $g$ its metric. 
Let 
\[\nabla:C^\infty(M,TM)\to C^\infty(M,\Lambda^1 M\otimes TM)\] 
denote the Levi-Civita connection, and 
\[
d^\nabla:  C^\infty(M,\Lambda^k M\otimes TM) \longrightarrow C^\infty(M,\Lambda^{k+1} M\otimes TM)
\]
its extension to $k$-forms twisted by vector fields. The square of this operator equals, in every degree $k$, 
the curvature operator $R$, viewed as an endomorphism-valued $2$-form:
\[
(d^\nabla)^2= R\in C^\infty(M,\Lambda^2 M\otimes\mathrm{End}(TM)).
\]
The second Bianchi identity is an immediate consequence:
\begin{equation}\label{dnc}
d^\nabla(R)=[d^\nabla, (d^\nabla)^2]=0.
\end{equation}
The \emph{Ricci tensor}
is the contraction on positions $\{1,4\}$ of $R$ viewed as a $(3,1)$ tensor:
\[
\Ric(U,V)=\tr\left[X\mapsto R_{XU}V \right].
\]
The \emph{scalar curvature} $\scal$ is the trace of $\Ric$ with respect to $g$.

Let $(d^\nabla)^*$ denote the formal adjoint of $d^\nabla$. This operator restricted to
symmetric $2$-tensors is sometimes called the \emph{divergence} operator.
By twice contracting the $(4,1)$ tensor from \eqref{dnc} on positions $\{3,4\}$ and then $\{1,5\}$,
we obtain
\begin{equation}\label{Pdivf}
(d^\nabla)^* \left(\Ric-\tfrac{\scal}{2}g\right)=0.
\end{equation}

\begin{definition}
The \emph{Schouten tensor} of the metric $g$ is essentially the Ricci tensor, namely
\[\Sch:=\Ric-\tfrac{\scal}{4}g.\] 
\end{definition}
Notice that since the dimension of $M$ is $3$, $\tr(\Sch)=\frac{\scal}{4}$. Also from \eqref{Pdivf},
$\Sch$ is not divergence-free unless $\scal$ is constant.

\begin{definition}
The \emph{Cotton form} of $g$ is $\cott:=d^\nabla \Sch\in C^\infty(M,\Lambda^2 M\otimes TM)$. 
Assuming $M$ to be oriented, the 
\emph{Cotton tensor} of $g$ is the bilinear form
\[\Cot:=*d^\nabla \Sch\in C^\infty(M,\Lambda^1 M\otimes TM),\]
where $*$ is the Hodge operator transforming $2$-forms into $1$-forms on $M$. 
\end{definition}
Recall that the Hodge operator in dimension $3$ is defined in terms of an oriented orthonormal frame
$S_1,S_2,S_3$ as follows:
\begin{align}\label{hoop}
*1=S^1\wedge S^2\wedge S^3,&& *S^1=S^2\wedge S^3,&& *S^2=S^3\wedge S^1,&& *S^3=S^1\wedge S^2
\end{align}
and satisfies $*^2=\mathrm{Id}_{\Lambda^*M}$.

\section{Chern-Simons forms of degree $3$}

Let $M$ be a smooth manifold of any dimension, $m\geq 1$ a natural number and $\theta$ a $1$-form
on $M$ with values in $\gl(m,\bR)$. The Chern-Simons form associated
to $\theta$ is defined by
\[
\cs(\theta):= \tr\left(\theta\wedge d\theta+\tfrac23 \theta\wedge\theta\wedge\theta\right)\in\Lambda^3(M).
\]
\begin{lemma}\label{varcs}
Let $\{\theta_t\}_{t\in\bR}$ be a smooth family of forms and denote by $\dot\theta:=\frac{d\theta_t}{dt}|_{t=0}$ 
its variation in $t=0$. Then 
\[\tfrac{d}{dt}\cs(\theta_t)_{|t=0}= d\tr(\dot\theta\wedge\theta)+2\tr(\dot\theta\wedge(d\theta+\theta\wedge\theta)).\]
\end{lemma}
\begin{proof}
For matrix-valued forms $\alpha,\beta\in\Lambda^*(M)\otimes \gl(m,\bR)$ we have the trace identity
\[\tr(\alpha\wedge\beta)=(-1)^{\deg\alpha\cdot\deg\beta}\tr(\beta\wedge\alpha)\]
although of course the matrix-valued forms $\alpha$ and $\beta$ need not commute.
We also use the obvious rule for taking the exterior differential of a trace, namely
\[d\tr(\alpha\wedge\beta)=\tr(d\alpha\wedge \beta+(-1)^{\deg(\alpha)}\alpha\wedge d\beta).\]
We compute by differentiating under the trace by the Leibnitz rule and using the above trace identity:
\begin{align*}
\tfrac{d}{dt}\cs(\theta_t)_{|t=0}={}& \tr(\dot\theta\wedge d\theta + \theta\wedge d\dot\theta + 2\dot\theta\wedge\theta\wedge\theta)\\
={}& \tr(d\dot\theta\wedge \theta -\dot\theta\wedge d\theta + 2\dot\theta\wedge d\theta+2\dot\theta\wedge\theta\wedge\theta)
\end{align*}
which gives the desired formula.
\end{proof}

Set $\Omega:=d\theta+\theta\wedge\theta\in\Lambda^2(M)\otimes \gl(m,\bR)$. Lemma \ref{varcs} can be rewritten
\[
\tfrac{d}{dt}\cs(\theta_t)_{|t=0}= d\tr(\dot\theta\wedge\theta)+2\tr(\dot\theta\wedge\Omega).
\]

\begin{lemma}\label{lem4}
The exterior derivative of $\cs(\omega)$ is 
\[d\cs(\theta)= \tr(\Omega\wedge\Omega).\]
\end{lemma}
\begin{proof}
We have by the trace identity
\begin{align*}
d\cs(\theta)= {}&\tr(d\theta\wedge d\theta +2 d\theta\wedge \theta\wedge\theta)= \tr\left( (d\theta+\theta\wedge\theta)^2\right)
\end{align*}
where in the last equality we used $\tr(\theta^4)=0$ by anti-symmetry.
\end{proof}

\section{The Chern-Simons invariant of a Riemannian metric}

From now on $M$ is an oriented compact $3$-manifold without boundary.
The tangent bundle to such a manifold
is trivial by a classical result of Stiefel \cite[Satz 21]{stiefel} 
(in fact, for this result one does not even need $M$ to be compact, 
but we do not use here this more general statement). 
Choose a global orthonormal frame $S=(S_1,S_2,S_3)$, obtained for instance by applying the Gram-Schmidt procedure to
some arbitrary smooth frame. In this frame, the Levi-Civita connection
can be expressed as
\[\nabla = d+\omega\]
where the $\sot$-valued connection $1$-form $\omega$ is given by 
$\omega_{ij}=\langle \nabla S_j,S_i\rangle $. 

\begin{definition}
The Chern-Simons 
$3$-form $\cs(M,g,S)$ is defined by
\[\cs(M,g,S):= \cs(\omega)=\tr(\omega\wedge d\omega+\tfrac23 \omega\wedge\omega\wedge\omega)\in\Lambda^3(M),\]
where the trace is the usual trace on $3\times 3$ matrices.  
The Chern-Simons invariant of $(M,g)$ with respect to the frame S is 
\[ \CS(M,g, S):= -\tfrac{1}{16\pi^2} \int_M \cs(M,g,S)\in \bR.\]
\end{definition}

The triviality of the vector bundle $TM$ entails the vanishing of its
``primary'' characteristic classes (Stiefel-Whitney, Pontriagin and Euler). 
At the same time, this triviality implies the existence of a global connection $1$-form, 
the essential ingredient
in the construction of the Chern-Simons invariant. Since the invariant 
exists under the condition that all primary characteristic classes vanish, it is viewed as a kind of 
``secondary'' characteristic class, dependent on the metric, and \emph{a priori} also 
on the choice of the frame. 

\begin{proposition}
Let $S,S'$ be two orthonormal frames on $M$ linked by some $\Sot$-valued function 
$\mfa:M\to\Sot$, i.e., $S'=S\mfa$. Then 
\[\CS(M,g, S)-\CS(M,g,S')\in\bZ.\]
\end{proposition}
\begin{proof}
The connection form changes by
\[\omega'=\mfa^{-1}\omega \mfa +\mfa^{-1} d\mfa.\]
By a simple computation, the Chern-Simons form of $\omega'$ is given by
\[\cs(\omega')=\cs(\omega)+d\tr(\mfa^{-1}\omega\wedge d\mfa)-\tfrac13 \tr((\mfa^{-1}d\mfa)^3).\]
The form $\mfa^{-1}d\mfa$ equals the pull-back via the map $\mfa:M\to \Sot$ of the Maurer-Cartan 
$1$-form $\omc$ that we recall in Lemma \ref{lemmc} below.
By that lemma, the integral on $\Sot$ of the $3$-form $\tr((\omega_{MC})^3)$ equals $48\pi^2$, thus
$\frac{1}{48\pi^2}\tr((\omc)^3)$ is a generator of $H^3(\Sot,\bZ)$. It follows that the integral 
on $M$ of $\frac{1}{48\pi^2}\tr((\mfa^{-1}d\mfa)^3)$ is
an integer, equal to the degree of the map $\mfa$. 
\end{proof}

We see that the constant of normalization was chosen so that for different choices of $S$,
the Chern-Simons integral changes by some integer. In other words, the invariant 
is well-defined modulo $\bZ$ independently of $S$, and will be denoted $\CS(M,g)\in\bR/ \bZ$.

\section{An example}\label{exampl}

Let us compute the Chern-Simons invariant of the Lie group $\Sot$. 
Let $\bH$ be the quaternion algebra. The group of quaternions of length $1$ (which is just
the sphere $S^3\subset \bH$) 
acts on $\bH$ via \emph{right} multiplication,
preserving at the same time the standard Hermitian metric and the structure of complex vector space 
given by left multiplication with complex numbers. We get in this way a unitary representation of $S^3$.
By compactness, the representation lies in $\Sud$, and since it is clearly faithful, it provides 
a Lie group isomorphism $\Sud\to S^3$.

Conjugation by quaternions of length $1$ is a real representation of $S^3$ (which we henceforth identify with
$\Sud$). It acts orthogonally on $\bH=\bR^4$ and preserves the real line, 
thus it also preserves its orthogonal complement $\bH'=\bR^3$, the $3$-dimensional space
of purely imaginary quaternions. By connectedness, it must take values in $\Sot$.
The kernel of this representation is the intersection of the center of $\bH$ with $S^3$, thus it consists
of $\{\pm 1\}$. Moreover the representation is surjective since it contains every reflection around an axis in $\bH'$.

We have obtained a $2:1$ covering
$\Sud\to\Sot$ of Lie groups, with deck group $\{\pm 1\}$ acting isometrically. 
Endow $\Sot$ with the metric $g$ induced from this covering.

Let 
\begin{align}\label{IJK}
I:=\ad_i=\begin{bmatrix}
     0&0&0\\0&0&-2\\0&2&0
    \end{bmatrix}, &&
J:=\ad_j=\begin{bmatrix}
     0&0&2\\0&0&0\\-2&0&0
    \end{bmatrix},&&
K:=\ad_k=\begin{bmatrix}
     0&-2&0\\2&0&0\\0&0&0
    \end{bmatrix}
\end{align}
be the image in $\sot$ of 
the standard orthonormal basis $\{i,j,k\}$ in $T_1 S^3$. We transport these vectors on $\Sot$ by left translations. Their Lie bracket is given by:
\begin{align*}
[I,J]=2K,&&[J,K]=2I,&& [K,I]=2J.
\end{align*}
From the Koszul formula \eqref{kos} and left-invariance, we deduce
\begin{align*}
\nabla_IJ=K,&&\nabla_JK=I,&& \nabla_KI=J.
\end{align*}
The connexion $1$-form $\omega$ is thus given by
\begin{align*}
\omega(I)=\begin{bmatrix}
     0&0&0\\0&0&-1\\0&1&0
    \end{bmatrix}, &&
\omega(J)=\begin{bmatrix}
     0&0&1\\0&0&0\\-1&0&0
    \end{bmatrix},&&
\omega(K)=\begin{bmatrix}
     0&-1&0\\1&0&0\\0&0&0
    \end{bmatrix}
\end{align*}
implying that $\tr(\omega^3)=-6 d\vol_g$. Since the coefficients of $\omega$ are constant,
it follows from the Cartan formula
\begin{align*}
d\omega(I,J)=-2\omega(K),&&d\omega(J,K)=-2\omega(I),&& d\omega(K,I)=-2\omega(J).
\end{align*}
Since $\tr\left(\omega(I)^2\right)=\tr\left(\omega(I)^2\right)=\tr\left(\omega(I)^2\right)=-2$,
it follows 
 $\tr(\omega\wedge d\omega)=12 d\vol_g$. In conclusion, 
 \[\cs(\omega)=8d\vol_g.\]
 
The volume of the sphere $S^3$ with its standard metric is obtained as follows:
\[
\vol(S^3)= \int_{-\pi/2}^{\pi/2} \cos^2(t) \vol(S^2) dt= 4\pi \cdot \tfrac{\pi}{2}=2\pi^2,
\]
so $\vol(\Sot,g)=\pi^2$. We deduce
\begin{equation}\label{cssot}
\CS(\Sot,g)=-\tfrac{1}{2}.
\end{equation}

\begin{corollary}\label{csst}
The Chern-Simons invariant of $S^3$ with its standard metric vanishes in $\bR/\bZ$.
\end{corollary}
\begin{proof}
The $2:1$ covering $S^3\to\Sot$ is an isometry, hence $\int_{S^3}\cs(S^3)=2\int_{\Sot}
\cs(\Sot)$. The result follows from \eqref{cssot}.
\end{proof}

\begin{lemma}\label{lemmc}
Let $\omc$ denote the Maurer-Cartan $1$-form on $\Sot$, namely 
$\omc(X)=X\in\sot$ for every left-invariant vector field $X$. Then 
\[\tr(\omc^3)=-48 d\vol_g.\]
\end{lemma}
\begin{proof}
It follows from \eqref{IJK} that
\begin{align*}
\tr(\omc(I)\omc(J)\omc(K))=\tr(IJK)={}&-8,\\
\tr(\omc(I)\omc(K)\omc(J))=\tr(IKJ)={}&8.
\end{align*}
From the trace identity, the lemma follows.
\end{proof}

\section{Conformal invariance}

One of the striking properties of the Riemannian Chern-Simons invariant is its conformal invariance: the invariant
does not change (modulo $\bZ$) when the metric varies in a fixed conformal class.

\begin{theorem}
Let $(M,g)$ be a closed oriented Riemannian $3$-fold, and $f\in C^\infty(M,\bR)$ an arbitrary conformal factor. Then
\[\CS(M,g)=\CS(M,e^{2f}g).\]
\end{theorem}
\begin{proof}
For $t\in\bR$ set $g^t:=e^{2tf}g$, thus in particular $g^0=g$. 
Fix an orthonormal frame $S=(S_1,S_2,S_3)$ on $(M,g)$, and define $S_j^t:=e^{-tf}S_j$, $j=1,2,3$.
Then $(S_1^t,S_2^t,S_3^t)$ form an orthonormal frame on $(M,g^t)$. Starting from the Koszul formula
\begin{equation}\label{kos}\begin{split}
2\langle\nabla_XY,Z\rangle= {}&X\langle Y,Z \rangle +Y\langle X,Z \rangle -Z\langle X,Y\rangle\\
{}& +\langle[X,Y],Z \rangle +\langle[Z,X],Y \rangle 
+\langle[Z,Y],X \rangle,
\end{split}\end{equation}
we deduce that the Levi-Civita connexions of $g^1=e^{2f}g$ and that of $g$ differ by
\[\nabla^1_XY - \nabla_XY = X(f) Y+ Y(f) X - \langle X,Y\rangle_{g}\mathrm{grad}^g(f).\]
for every vector fields $X,Y$. We apply this identity to deduce
\[\langle \nabla^1_X S^1_j,S^1_i\rangle_{g^1} =\langle \nabla_X S_j,S_i\rangle 
+ \langle S_j(f) S_i,X\rangle - \langle S_i(f) S_j,X\rangle.
\]
In other words, if we identify vectors and $1$-forms using the metric $g$, we express the 
connection $1$-form of $\nabla^1$ in the frame $S^1$ as
\[\omega^1 = \omega +\alpha\]
where 
\[\alpha_{ij}(X) = \langle S_j(f) S_i -  S_i(f) S_j,X\rangle.\]

Applying this identity to $g^t$, we get $\omega^t=\omega+t\alpha$. 
From Lemma \ref{varcs} with $\dot{\omega}=\alpha$,
we obtain 
\begin{equation}\label{vcs}
\tfrac{d}{dt}\cs(\omega^t)_{|t=0}= d\tr(\alpha\wedge\omega)+2\tr(\alpha\wedge R)
\end{equation}
where $R\in\Lambda^2(M,\sot)$ is the Riemannian curvature tensor written in the frame 
$S$.
\begin{lemma}
The trace $\tr(\alpha\wedge R)$ vanishes.
\end{lemma}
\begin{proof}
Let $(k,h,l)\in\Sigma(3)^+$ denote even permutations of $\{1,2,3\}$.
We write 
\begin{align*}
\tr(\alpha\wedge R) = d\vol_g \sum_{(k,h,l)\in\Sigma(3)^+} \sum_{i,j=1}^3 \alpha_{ji}(S_k)R_{hlij}.
\end{align*}
But $\alpha_{ji}(S_k)= S_i(f)\delta_{kj} - S_j(f)\delta_{ki}$, hence
\[\tr(\alpha\wedge R)= 
d\vol_g \sum_{(k,h,l)\in\Sigma(3)^+} \left( \sum_{i=1}^3 S_i(f) R_{hlik} - \sum_{j=1}^3 S_j(f) R_{hlkj} \right).
\]
Both terms vanish by the first Bianchi identity (see \cite[Prop.\ 1.85d]{besse}).
\end{proof}
By this lemma, \eqref{vcs} can be rewritten
\[\tfrac{d}{dt}\cs(\omega^t)_{|t=0}= d\tr(\alpha\wedge\omega).\]
Now apply this to $\omega_t$ instead of $\omega$.
We have seen that $\frac{d\omega^t}{dt}=\alpha$ is independent of $t$. 
Therefore the identity
\[
\tfrac{d}{dt}\cs(\omega^t)= d\tr(\alpha\wedge\omega^t)
\]
is valid for all $t$. Moreover, $\tr(\alpha\wedge\alpha)=0$ and so finally we deduce
$\frac{d\cs(\omega^t)}{dt}= d\tr(\alpha\wedge\omega)$. This is independent of $t$, and by integrating from $0$ to $1$
we get $\cs(\omega^1)=\cs(\omega)+d\tr(\alpha\wedge\omega)$. 
The theorem is a direct consequence of Stokes' formula.
\end{proof}

\section{First-order variation of the Chern-Simons invariant}

Recall the definitions of the Cotton form, respectively tensor, on a Riemannian $3$-manifold $(M,g)$:
\begin{align*}
\cott=d^\nabla \Sch,&&\Cot=*\cott.
\end{align*}
Let $L^2$ denote the Hilbert space of square-integrable symmetric $2$-tensors on $M$ with 
respect to the scalar product and the volume form induced by $g$.
\begin{theorem}\label{varcsco}
Let $\Omega:=d\omega+\omega\wedge\omega\in \Lambda^2(M,\sot)$ denote the curvature form of $g$
in the frame $S$. Let $(g^t)_{t\in\bR}$ be a $1$-parameter family of Riemannian metrics, and denote by
$\dg:=\frac{d}{dt} {g^t}_{|t=0}$ its first-order variation.
Then 
\[\dCS:=\tfrac{d}{dt}\CS(M,g^t)_{|t=0}= -\tfrac{1}{8\pi^2}\langle \dg, \Cot\rangle_{L^2}=
-\tfrac{1}{8\pi^2}\int_M \langle \dg,\Cot(g)\rangle_g d\vol_g.  \]
\end{theorem}
\begin{proof}
We first construct a smooth family $(S^t)_{t\in\bR}$ of global frames on $M$ such that $S^t$ is orthonormal for $g^t$.
For this, consider the metric $g^\cZ:=dt^2+g^t$ on $\cZ:=\bR\times M$, and set $S^t$ to be the parallel transport 
along the segment $[0,t]$ of the frame $S^0:=S$. A direct computation shows that the lines $\bR\times\{p\}$ are geodesics
for $g^\cZ$ for all $p\in M$, thus $S^t$ is indeed an orthonormal frame in the slice $\{t\}\times M$. 

Let $\omega^t$ be the connection $1$-form of $g^t$ in the frame $S^t$, and set 
$\dot \omega:=\frac{d\omega^t}{dt}|_{t=0}$. 
Applying Lemma \ref{varcs} and interchanging the integral on $M$ with the $t$-differential, we find 
\[-16\pi^2 \dCS=2 \int_M\tr(\dot \omega\wedge \Omega).\]
Let $X\in \cV(M)$ be a vector field on $M$, extended on $\cZ$ to be constant in the $t$ direction, in other words
$L_TX=0$ or equivalently $[T,X]=0$ where we denote $T:=\frac{\partial}{\partial t}$. Then
\begin{align*}
\dot\omega_{ij}(X)={}&\partial_t g^t(\nabla^t_X S^t_j,S^t_i)= \partial_t g^\cZ(\nabla^\cZ_X S^t_j,S^t_i)
=\langle \nabla^\cZ_T\nabla^\cZ_X S^t_j,S^t_i\rangle = \langle R^\cZ_{TX} S^t_j, S^t_i\rangle.
\end{align*}
We used above the fact that $\nabla^\cZ_T S^t_j=0$ (by construction of $S^t_j$) and the commutation of 
$T$ and $X$. From the symmetries of the Riemannian curvature, we get 
\[\dot\omega_{ij} = R^\cZ_{S_j, S_i} T\]
(we identify vectors and $1$-forms using the metric $g$). Let $W$ be 
the Weingarten operator of $\{0\}\times M\hookrightarrow \cZ$. By the Codazzi-Mainardi equation 
\cite[1.72d]{besse}, 
\[R^\cZ_{S_j, S_i} T = d^\nabla W (S_j, S_i)\]
or equivalently $R^\cZ T=d^\nabla W$ as vector-valued $2$-forms.
The operator $W$, essentially the second fundamental form, can be computed 
in terms of the first variation $h$ of $g^t$, namely
\[ W= g^{-1}\II = \tfrac12 g^{-1}\dg.
 \]
It follows that $\dot\omega = \tfrac12 d^\nabla \dg$, where $d^\nabla \dg$ is viewed as a $\sot$-valued
$1$-form using the basis $S$. Since the trace is independent of the basis,
\[-16\pi^2 \dCS= \int_M \tr(d^\nabla \dg \wedge R)\]
where $R$ is the Riemannian curvature tensor viewed as a section in
$\Lambda^2(M)\otimes \mathrm{End}(TM)$, and $d^\nabla \dg$ is viewed as a section in
$\mathrm{End}(TM)\otimes \Lambda^1(M)$. Using  the symmetry of $R$,
this becomes
\[-2\int_M \langle d^\nabla \dg, *_{34} R\rangle d\vol_g,\]
where the Hodge star operator $*_{34}$ acts on the last two positions in $R$. 
By definition of the adjoint operator, this equals
\[-2\langle \dg, {d^\nabla}^* *_{34} R\rangle_{L^2}.\]
Now the adjoint of $d^\nabla$ on $2$-forms with values in $T^*M$ is just $-*_{12} d^\nabla *_{12}$, thus 
\[-16\pi^2 \dCS=2\langle \dg, *_{12} {d^\nabla} *_{12}*_{34} R\rangle_{L^2}.\] 
The $2$-tensor $*_{12}*_{34} R$ is easily computed:
\[*_{12}*_{34} R = \Ric-\tfrac{\scal}{2}g=:Q=\Sch-\tfrac{\scal}{4}g.\]
We remark that $h$ is symmetric but ${d^\nabla}^* *_{34} R$ is not necessarily so. 
Of course, the skew-symmetric component
of ${d^\nabla}^* *_{34} R$ will not contribute towards $\dot{\CS}$ since $h$ is symmetric. Thus, 
\[-8\pi^2 \dCS=\langle \dg, (*_{12} {d^\nabla} *_{12}*_{34} R)_{\sym}\rangle_{L^2}
= \langle \dg, (*_{12} {d^\nabla}Q)_{\sym}\rangle_{L^2}. \]
We now claim that $(*_{12} {d^\nabla}Q)_{\sym}=*_{12} {d^\nabla}\Sch=\Cot$. 
This means two things: 
\begin{itemize}
\item $*_{12} {d^\nabla}\Sch$ is symmetric;
\item $*_{12} {d^\nabla} (\scal\cdot g)$ is skew-symmetric.
\end{itemize}
For any function $f$, $*_{12} {d^\nabla} (f g)= * df$ is a $2$-form, so the second fact 
is evident. As for the first, we use the identity \cite[1.94]{besse}
\[ \tr_{13}d^\nabla \Sch={d^\nabla}^*(\Ric-\tfrac{\scal}{2}g)=0,\]
where $\tr_{13}$ denotes trace with respect to $g$ on positions $1,3$. 
Take $(i,j,k)$ to be a cyclic permutation of $(1,2,3)$. Then 
\[0=(\tr_{13}d^\nabla \Sch)(S_k)= d^\nabla \Sch(S_i,S_k,S_i)+d^\nabla \Sch(S_j,S_k,S_j),\]
which is equivalent (by immediate algebraic considerations) to the desired symmetry
\[
*_{12} {d^\nabla}\Sch(S_i,S_j) = *_{12} {d^\nabla}\Sch(S_j,S_i).
\]
\end{proof}

\section{Properties of the Cotton tensor}

Let $(M^\circ,g^\circ)$ be a Riemannian $3$-manifold, not necessarily compact.
Then every point $x\in M^\circ$ has a neighborhood which can be isometrically embedded
in some compact manifold $(M,g)$. We will prove below some local properties of the Cotton tensor 
of $(M,g)$, which are therefore shared by $\Cot(M^\circ,g^\circ)$.

\begin{proposition}\label{cotsym}
The Cotton tensor $*{d^\nabla}\Sch$ is symmetric.
\end{proposition}
\begin{proof}
This was shown in the last part of the proof of Theorem \ref{varcsco}.
\end{proof}

\begin{proposition}
The Cotton tensor is trace-free.
\end{proposition}
\begin{proof}
Let $f\in C^\infty(M,\bR)$ be arbitrary, and set $g^t:=e^{2tf}g$. Its first-order variation at $t=0$
is given by $2f g$.
By conformal invariance, $\CS(M,g^t)$ is constant in time. On the other hand,
by Theorem \ref{varcsco}, we get
\[0= 16\pi^2 \frac{d\CS(M,g^t)}{dt}|_{t=0} = \langle 2f g,\Cot(g)\rangle=2\int_M f\tr(\Cot(g))d\vol_g.\]
This means that $\tr(\Cot(g))$ is $L^2$-orthogonal on every smooth function on $M$, hence it must vanish identically.
\end{proof}

\begin{proposition}
The Cotton tensor is divergence-free, i.e., ${d^\nabla}^*\Cot(g)=0$.
\end{proposition}
\begin{proof}
Let $X\in\cV(M)$ be a vector field, and $\phi_t$ the $1$-parameter group of diffeomorphisms obtained by integrating $X$ on $M$.
Set $g^t:=\phi_t^*g$, so in particular $\dot{g}=L_Xg$. Since $g$ and $g^t$ are isometric, it is rather evident that
$\CS(M,g^t)=\CS(M,g)$. By Theorem \ref{varcsco} we deduce that $\langle L_X g, \Cot(g)\rangle_{L^2}=0$.
But $L_Xg =\tfrac12 (d^\nabla X)_\sym$, so by Proposition \ref{cotsym} and the definition of the adjoint,
$\langle L_X g, \Cot(g)\rangle= \tfrac12 \langle X, {d^\nabla}^*\Cot\rangle$.
Hence ${d^\nabla}^*\Cot$ is $L^2$-orthogonal to every vector field $X$, so it must vanish identically.
\end{proof}

\begin{proposition}\label{confcotin}
The Cotton tensor is conformally covariant, in the sense that for every $f\in C^\infty(M)$, we have
\[\Cot(e^{2f}g)=e^{-f}\Cot(g).\]
The Cotton form $\cott=d^\nabla \Sch$ is conformally invariant:
\[\cott(e^{2f}g)=\cott(g).\]
\end{proposition}
\begin{proof}
Let $h$ be any symmetric $2$-tensor, and choose a family $g^t$ of metrics with $\dot{g}=h$, for instance 
$g^t=g+th$ for small $t$. By conformal invariance of the Chern-Simons invariant, 
$\CS(M, e^{2f}g^t)=\CS(M, g^t)$ so their first variation in $t=0$ must be equal.
Therefore by Theorem \ref{varcsco}, $\langle e^{2f} h,\Cot(e^{2f}g)\rangle_{L^2(e^{2f}g)}= \langle  h,\Cot(g)\rangle_{L^2(g)}$, or equivalently
\[\langle e^{f}h, \Cot(e^{2f}g)\rangle_{L^2(g)}= \langle h, \Cot(g)\rangle_{L^2(g)}.
\]
Since $h$ was arbitrary, we deduce that $e^f \Cot(e^{2f}g)=\Cot(g)$. Using the obvious rescaling properties of the Hodge star operator
under conformal transformations in dimension $3$ acting on $1$-forms, we deduce that the Cotton forms of $g$ and $e^{2f}g$ are equal.
\end{proof}

These four propositions can of course be proved directly from the definitions, by local computations. 
We hope nevertheless that the reader will admit the qualitative advantage of proving properties of the Cotton tensor 
via Chern-Simons invariants. The difficulty of the proof was hidden in the definition and properties 
of the latter, but in exchange we gain a superior insight for the properties of the former.

\section{Conformal immersions in $\bR^4$}

\begin{theorem}
Let $(M,g)$ be a closed oriented Riemannian $3$-manifold and assume there exists a conformal immersion
$\imath:M\to\bR^4$. Then the Chern-Simons invariant $\CS(M,g)$ vanishes.
\end{theorem}
\begin{proof}
Since the Chern-Simons invariant is conformally invariant, by replacing $g$ with $\imath^*g^{\bR^4}$
we can assume that $\imath$ is an isometric immersion.
Let $N:M\to S^3$ be the Gauss map of the immersion, i.e., $N(x)$ is the unit normal to $\imath_*(T_xM)$ chosen such that
if $(S_1,S_2,S_3)$ is an oriented frame in $T_xM$, then $(N(x),S_1,S_2,S_3)$ is positively oriented in $\bR^4$.
We identify $\bR^4$ with the quaternion algebra as in Section \ref{exampl}. For every point $\eta\in S^3$, consider the
orthonormal frame in $T_\eta S^3$
\begin{align*}
U_1=i\eta,&&U_2=j\eta,&&U_3=k\eta.
\end{align*}
Similarly, for every $x\in M$ consider the orthonormal frame in $T_xM$
\begin{align*}
S_1=\imath_*^{-1}(iN(x)),&&S_2=\imath_*^{-1}(jN(x)),&&S_3=\imath_*^{-1}(kN(x)).
\end{align*}
We claim that in these frames, the connection $1$-forms on $S^3$ and on $M$ are related by
\begin{equation}\label{tiom}
N^*\omega^{S^3}=\omega^M.
\end{equation}
This is a local statement so we can assume that $M$ is a hypersurface in $\bR^4$. 
Notice that for every $x\in M$, we have $S_i(x)=U_i(N(x))$, $i=1,2,3$.
Let $X\in T_xM$ 
be a vector tangent to a curve $\gamma$ in $M$ with $x=\gamma(0)$. Then using the definition
of the Levi-Civita connection for hypersurfaces, we get
\begin{align*}\langle \nabla^M_X S_1,S_2\rangle={}& \langle \pt S_1(\gamma(t)),S_2\rangle=
\langle i\pt N(\gamma(t)),jN_x\rangle,\\
\langle \nabla^{S^3}_{N_*X} U_1,U_2\rangle={}& \langle \pt U_1(N(\gamma(t))),U_2\rangle=
\langle i\pt N(\gamma(t)),jN_x\rangle
\end{align*}
so $\omega^M_{21}(X)=\omega^{S^3}_{21}(N_* X)$. The same argument for the other pairs of indices
end the proof of \eqref{tiom}. Then by definition, the Chern-Simons forms on $M$ and $S^3$
are related by 
\[N^*\cs(S^3,U)=\cs(M,S).\]
It follows that $\CS(M,g)=\deg(N)\CS(S^3)$ is an integer, by Corollary \ref{csst}
(where $\deg(N)\in \bZ$ is the topological degree of the Gauss map $N:M\to S^3$).
\end{proof}

In particular, it follows from this theorem and \eqref{cssot} that $\Sot$ cannot be conformally immersed in $\bR^4$, 
although it is locally isometric to $S^3\subset\bR^4$.

\section{Locally conformally flat metrics}

We end our excursion into Cotton territory by solving the lcf (locally conformally flat) problem in dimension $3$:
\begin{quote}
\emph{When does a Riemannian metric $g$ on a $3$-manifold $M$ admit, 
for every $x\in M$, a conformal factor $f\in C^\infty(V)$ defined on some neighborhood $V\ni x$ 
such that $e^{2f}g$ is flat?}
\end{quote}

From Proposition \ref{confcotin}, one necessary condition for a positive answer is the vanishing
of the Cotton tensor $\Cot=*d^\nabla \Sch$, 
where $\Sch$ is the Schouten tensor.


\begin{theorem} \cite{cotton}
A Riemannian metric $g$ on a $3$-manifold $M$ is locally conformally flat 
if and only if its Cotton tensor $\Cot(g)$ vanishes. 
\end{theorem}
The Cotton tensor and the Cotton form are obtained from one another through the Hodge star, so they vanish
simultaneously. In dimension $3$ the Schouten tensor determines completely the curvature tensor since
for every vectors $U,V,X$ we have (cf.\ \cite[1.119b]{besse})
\begin{equation}\label{Rsch}
R_{U,V}X= \langle X,V\rangle \Sch(U) + \langle \Sch(X),V\rangle U 
- \langle \Sch(X),U\rangle V-\langle U,X\rangle \Sch(V),
\end{equation}
or in other words 
\begin{equation}\label{exkn}
RX=-X\wedge \Sch -\Sch(X)\wedge g. 
\end{equation}
Moreover, $\Sch$ and $\Ric$ determine one another, since $\tr(\Sch)=\scal(g)/4$. Hence $g$ 
is conformally flat if and only if it is conformally Schouten-flat.
\begin{proof}
Take $g_1=e^{2f}g$ for some $f\in C^\infty(M)$. We know by Proposition \ref{confcotin}
that $\cott=\cott_1$. Thus, if $\cott_1=0$ if follows that $\cott=0$.

Conversely, recall from \cite[1.159]{besse} the formula for the conformal change of the Schouten tensor: 
if we set $X:=df$,
\begin{equation}\label{conftrsc}
\Sch-\Sch_1= \nabla X -X\otimes X+\tfrac{1}{2}|X|^2 g.
\end{equation}
Suppose that we can solve locally the equation
\begin{equation}\label{conftrsceq}
\Sch= \nabla X -X\otimes X+\tfrac{1}{2}|X|^2 g.
\end{equation}
with some vector field $X$ (identified to a $1$-form using $g$). Then the term $\nabla X$ 
is symmetric, hence $X$ must be closed. From the Poincar\'e lemma, $X$ is locally exact, 
thus there exists (locally) a function $f$ with $X=df$. By combining \eqref{conftrsc} and 
\eqref{conftrsceq}, the metric $g_1:=e^{2f}g$ will be Schouten-flat. Thus, in order to finish 
the proof it is enough to show that the equation \eqref{conftrsceq} is locally solvable
under the assumption that $\cott=0$.

Rewrite \eqref{conftrsceq} as the overdetermined system
\begin{equation}\label{ecd}
\nabla X= \Sch +X\otimes X-\tfrac{1}{2}|X|^2 g
\end{equation}
and apply the twisted exterior differential $d^\nabla$ in both sides. We get
\begin{equation}\label{ect}
RX= \cott +d^\nabla(X\otimes X-\tfrac{1}{2}|X|^2 g)
\end{equation}
where $R$ is the curvature tensor. Using \eqref{ecd}, the fact that $X$ is closed and $g$ is parallel, we compute
\begin{align*}
d^\nabla(X\otimes X-\tfrac{1}{2}|X|^2 g)={}& -X\wedge\nabla X -\langle \nabla X,X\rangle\wedge g\\
={}&-X\wedge \Sch +\tfrac{1}{2}|X|^2 X\wedge g\\
{}&-\Sch(X)\wedge g -|X|^2 X\wedge g +\tfrac{1}{2}|X|^2 X\wedge g
\end{align*}
which, substituting in \eqref{ect} and using the assumption $\cott=0$, reduces to the constraint
\[RX={}-X\wedge \Sch-\Sch(X)\wedge g\]
already noted above \eqref{exkn}. Thus the system \eqref{conftrsceq} is \emph{involutive}, 
so by the Frobenius theorem it is locally integrable. For completeness, let us prove this integrability by hand, 
without invoking the Frobenius theorem.

Choose local coordinates $x_1,x_2,x_3$ in $M$, and for $j=1,2,3$ denote by 
$\partial_j=\frac{\partial}{\partial x_j}$ the coordinate vector fields.
We fix $X(0)$, and extend $X$ along the axis $\{x_2=x_3=0\}$
in a neighborhood of the origin by solving the equation \eqref{ecd} in the direction of $\partial_1$:
\begin{equation}\label{A=0}
\nabla_{\partial_1} X= \Sch(\partial_1) +\langle X,\partial_1\rangle X-\tfrac{1}{2}|X|^2 \partial_1.\end{equation}
This is an ODE with smooth coefficients, hence the solution $X(t,0,0)$ exists for small time $t$ 
and is uniquely determined by the initial value $X(0)$. Now for every $t$, 
extend $X$ along the lines $\{x_1=t,x_3=0\}$ using the ODE obtained from \eqref{ecd} in the direction of $\partial_2$:
\begin{equation}\label{A=0'}
\nabla_{\partial_2} X= \Sch(\partial_2) +\langle X,\partial_2\rangle  X-\tfrac{1}{2}|X|^2 \partial_2.
\end{equation}
Again, the solution $X(t,s,0)$ exists for small time $s$, depends smoothly on the parameter $t$ and on the variable $s$,
and is uniquely determined by the values of $X$ in $(t,0,0)$. Finally, we extend $X$ along the lines $\{x_1=t,x_2=s \}$
using \eqref{ecd} in the direction of $\partial_3$. This defines a smooth vector field $X$ in a neighborhood of the 
origin, but we must still prove that \eqref{ecd} is satisfied. 

By construction we know \eqref{A=0} only at points of the form $(t,0,0)$, and \eqref{A=0'} only on the plane ${x_3=0}$. 
Let us prove that \eqref{A=0} holds on the plane $\{x_3=0\}$. For this, set 
\begin{align}
A:={}&\nabla_{\partial_1} X-(\Sch(\partial_1) 
+\langle X,\partial_1\rangle X-\tfrac{1}{2}|X|^2 \partial_1).\label{defA}
\end{align}
Since \eqref{A=0} is valid at $(t,0,0)$ we know that $A(t,0,0)=0$ for every $t$.

The fact that the system \eqref{conftrsceq} is involutive should imply that $A$ satisfies 
a linear system of ODE's in the direction of $x_2$. Explicitly we compute using \eqref{A=0'} repeatedly:
\begin{align*}
\nabla_{\partial_2}A ={}& \nabla_{\partial_2}\nabla_{\partial_1} X -\nabla_{\partial_2} \Sch(\partial_1)
-\partial_2\langle X,\partial_1\rangle X - \langle X,\partial_1\rangle \nabla_{\partial_2} X\\
{}&+ \langle \nabla_{\partial_2} X,X\rangle \partial_1 +\tfrac{1}{2}|X|^2 \nabla_{\partial_2} \partial_1.
\end{align*}
Now use $\nabla_{\partial_2}\nabla_{\partial_1} X= R_{\partial_2\partial_1}X
+\nabla_{\partial_1}\nabla_{\partial_2} X$ and $d^\nabla\Sch=\cott=0$. We get from \eqref{A=0'}
\begin{align*}
\nabla_{\partial_2}A ={}& R_{\partial_2\partial_1}X
+\nabla_{\partial_1}\left(\Sch(\partial_2) +\langle X,\partial_2\rangle  X-\tfrac{1}{2}|X|^2 \partial_2\right)
-\nabla_{\partial_2} \Sch(\partial_1)\\
{}&-\partial_2\langle X,\partial_1\rangle X - \langle X,\partial_1\rangle \nabla_{\partial_2} X
+ \langle \nabla_{\partial_2} X,X\rangle \partial_1 +\tfrac{1}{2}|X|^2 \nabla_{\partial_2} \partial_1\\
={}& R_{\partial_2\partial_1}X+dX(\partial_1,\partial_2)X +\langle X,\partial_2\rangle \nabla_{\partial_1} X
- \langle X,\partial_1\rangle \nabla_{\partial_2} X \\
{}&- \langle \nabla_{\partial_1} X,X\rangle \partial_2
+ \langle \nabla_{\partial_2} X,X\rangle \partial_1.
\end{align*}
Substitute $\nabla_{\partial_2} X$ and $\nabla_{\partial_1} X$ in the equation above, using \eqref{A=0'} and
\eqref{defA}:
\begin{align*}
\nabla_{\partial_2}A ={}& R_{\partial_2\partial_1}X +dX(\partial_1,\partial_2)X +\langle X,\partial_2\rangle A 
-\langle A,X \rangle \partial_2\\
{}& + \langle X,\partial_2\rangle \Sch(\partial_1)- \langle X,\partial_1\rangle \Sch(\partial_2)\\
{}& -\langle \Sch(\partial_1),X\rangle \partial_2+\langle \Sch(\partial_2),X\rangle \partial_1\\
={}&dX(\partial_1,\partial_2)X +\langle X,\partial_2\rangle A -\langle A,X \rangle \partial_2.
\end{align*}
where in the last equality we have used \eqref{Rsch}. From \eqref{defA}, \eqref{A=0'} and 
the symmetry of the Schouten tensor,
\[
dX(\partial_1,\partial_2)=\langle \nabla_{\partial_1}X,\partial_2\rangle - \langle \nabla_{\partial_2}X,\partial_1\rangle
=\langle A,\partial_2\rangle
\]
Hence we get
\begin{equation} \label{na2a}
\nabla_{\partial_2} A = \langle A,\partial_2\rangle +\langle X,\partial_2\rangle A -\langle A,X \rangle \partial_2
= L(A)
\end{equation}
where $L$ is an endomorphism of $TM$. Therefore the vector field $A$ is a solution of the linear ODE \eqref{na2a}
in the variable $x_2$, 
with zero initial values, hence it vanishes identically.

By precisely the same argument applied to the pairs of variables $\{x_1,x_3\}$ and $\{x_2,x_3\}$, we see that
\eqref{A=0} and \eqref{A=0'} continue to hold at every point $(t,s,x_3)$. This means that $X$ is a solution 
to \eqref{ecd} as claimed.
\end{proof}

From the proof, it appears that $X$ is closed, and uniquely determined by the initial value $X(0)$.
Therefore the conformal factor $f$, the primitive of $X$, is uniquely determined by four parameters,
its value and its differential at $0$. These $4$ degrees of freedom arise from the fact that locally on $\bR^3$, 
the Lie algebra of conformal Killing vector fields has dimension $10$, while the subalgebra 
of Killing vector fields has dimension $6$.

\section{Links with other invariants}

We survey below two beautiful mathematical objects related to the Chern-Simons invariant. 
This section is of course no longer self-contained.
\subsection{The eta invariant}

The Chern-Simons invariant is strongly related
to the \emph{eta invariant} of the odd signature operator on $M$. The eta invariant is
a real-valued invariant of closed oriented $3$-manifolds, initially introduced by Atiyah, Patodi and Singer 
\cite{apsii} as a correction
term in Hirzebrich's signature formula on a $4$-manifold with boundary. It can be defined for 
every elliptic self-adjoint differential operator $A$ acting on the sections of some vector bundle $E$, 
but its construction is non-elementary: one needs to understand the 
\emph{spectrum} of $A$, which is the discrete subset of $\bR$ of eigenvalues (with multiplicity) of $A$
viewed as a self-adjoint unbounded operator in $L^2(M,E)$. The non-zero part of the spectrum is denoted 
$\mathrm{Spec}(A)^*$.

The eta function, a meromorphic function in the variable $z\in\bC$, is defined for
$\Re(z)>3$ by the absolutely convergent series
\[
\eta(A;z)=\dim\ker(A)+\sum_{\lambda\in\mathrm{Spec}(A)^*} \sign(\lambda) |\lambda|^{-z}.\]
The eigenvalues of $A$ grow sufficiently fast to ensure absolute convergence in a half-plane, 
for instance if $A$ is of order $1$, then the series defining $\eta(A,z)$ is absolutely convergent for 
$\Re(z)>3$. The function thus obtained extends to $\bC$ with possible simple poles in $z\in\{2,-2,-4,-6,\ldots\}$, 
in particular $z=0$ is a regular point, and $\eta(A)$ is by definition that regular value.

When $A$ is the self-adjoint odd signature operator acting on $\Lambda^1(M)\oplus \Lambda^3(M)$, 
\[A:=*d-d*\]
(here $*$ is the Hodge star defined in \eqref{hoop}), the resulting 
eta invariant is denoted $\eta(M,g)$ to highlight its dependence on the metric.

\begin{theorem}[\cite{apsii}]\label{etacs}
Modulo $\bZ$, the following equality holds:
\[3\eta(M,g)\equiv 2\CS(M,g) \mod\bZ.\]
\end{theorem}
The proof relies on the signature formula of 
Atiyah, Patodi and Singer \cite{apsi} on a oriented $4$-manifold $X$ with boundary $M$:
\[
\signature(X)= -\tfrac{1}{24\pi^2}\int_X\tr((R^X)^2) - \eta(M,g).
\]
Here $\signature(X)\in\bZ$ is the signature of the intersection form on the relative cohomology $H^2(X,M;\bR)$, 
defined as the difference of the dimensions of maximal subspaces in $H^2(X,M;\bR)$ along which
the intersection form is positive, respectively negative definite.
The metric $g^X$ on $X$ is of product type near the boundary, in the sense that $L_\nu g^X=0$ for 
$\nu$ the geodesic vector field with respect to $g^X$ orthogonal to $M$
(here $L_\nu$ denotes Lie derivative). Moreover, $g^X$ restricts to $g$ on $M$.

We only care about the right-hand side of the signature formula modulo integers:
\begin{equation}\label{cse1}
\eta(M,g)+ \tfrac{1}{24\pi^2}\int_X\tr((R^X)^2)\in \bZ.
\end{equation}
Recall that for every orthonormal frame $S$ on $M$, we have by definition 
\begin{equation}\label{cse2}
\CS(M,g)+\tfrac{1}{16\pi^2}\int_M\cs(M,g,S)\in \bZ.
\end{equation}
To give the idea of the proof of Theorem \ref{etacs}, suppose that we can find $X$ a compact oriented four-manifold
bounded by $M$ (this is always possible by a result of Thom \cite{thom} about the oriented cobordism ring). 
Extend $g$ to a metric $g^X$ 
on $X$, and suppose that $S$ completed with the inner unit normal vector field can be extended to a frame $S^X$ on $X$
(this is \emph{not} always possible, for topological reasons). 
Nevertheless, whenever these assumptions are fulfilled, write using Stokes' formula and Lemma
\ref{lem4}
\[\int_M\cs(M,g,S) = \int_X\tr((R^X)^2).\]
This equality holds for instance when $X$ is a cylinder, with diffeomorphic boundary components $(M,g)$, $(M',g')$
with opposite orientations. Keeping $g'$ fixed and using \eqref{cse1}, \eqref{cse2}, we see that 
$\frac{3}{2}\eta(M,g)- \CS(M,g)$ is constant (modulo $\bZ$) on the space of Riemannian metrics
on $M$. That constant is shown in \cite{apsii} to be either $0$ or $\frac{1}{2}$, according to the parity of 
$\signature(X)$ if we choose the filling manifold $X$ to be Spin 
(using the vanishing of the Spin cobordism group in dimension $3$).

\subsection{The Selberg zeta function}

Assume $(M,g)$ is hyperbolic, i.e., its sectional curvatures are constant and equal to $-1$. 
As above, $M$ is supposed to be compact and orientable. 
Then $M$ is isometric to a quotient $\Gamma\backslash\bH^3$, where $\bH^3$ is the hyperbolic $3$-space,
and $\Gamma\subset\Psldc$ is a discrete subgroup of oriented isometries 
(i.e., a Kleinian group) consisting only of loxodromic elements: 
every non-trivial element of $\Gamma$ is conjugated in $\Psldc$ to a matrix of the form
$\begin{bmatrix}
v_\gamma&0 \\ 0& v_\gamma^{-1}
 \end{bmatrix}$ 
with $|v_\gamma|<1$, where of course $v_\gamma$ and its inverse are the eigenvalues of the matrix 
$\gamma$. The complex number $q_\gamma:=v_\gamma^2$ is called the \emph{multiplier}
of $\gamma$, corresponding to the fact that the action of $\gamma$ on the Riemann sphere 
(the ideal boundary of $\bH^3$) is conjugated to the multiplication by $q_\gamma$.
 
Closed geodesics in $M$ are in one-to-one correspondence with (non-trivial) 
conjugacy classes in $\Gamma\simeq\pi_1(M)$. 
A geodesic $c_{\gamma}$ coresponding to a conjugacy class $[\gamma]$ determines 
thus the multiplier $q_\gamma=e^{-(l_\gamma+i\theta_\gamma)}$. Viewed geometrically, 
$l_\gamma$ is the length of $c_\gamma$, while $e^{i\theta_\gamma}$ is the holonomy
along $c_{[\gamma]}$. Both quantities are expressible in terms of the trace 
$\tr(\gamma)=v_\gamma+v_\gamma^{-1}\in\bC$.

The \emph{Selberg zeta function of odd type} was defined by Millson 
\cite{millson} as an infinite product over the set $\cP$ of
primitive conjugacy classes
of $\Gamma$ (an element in $\Gamma$ is said to be primitive if it is not a nontrivial power of another 
element). The definition is a particular case of a construction from
Selberg's foundational paper \cite{selberg}:
\[
\cZ_\Gamma(\lambda)=\prod_{[\gamma]\in \cP} 
\prod_{m,n=0}^\infty\ 
\frac{1- q_\gamma^{m}(\overline{q}_\gamma)^{n+1} e^{-\lambda l_\gamma}}
{1- q_\gamma^{m+1}(\overline{q}_\gamma)^n e^{-\lambda l_\gamma}}.
\]

The product is absolutely convergent in the half-plane $\{\Re(\lambda)>0\}$, 
and has a meromorphic extension 
to the whole complex plane. Like the Riemann zeta function, the Selberg zeta function 
displays a symmetry around $0$ (moreover, it is known to have zeros only on the imaginary axis). The central value 
$\cZ_\Gamma(0)$ is given heuristically by the divergent product
\[
\cZ_\Gamma(0)= \prod_{[\gamma]\in\cP}\prod_{n\geq 1} \frac{1-(\overline{q}_\gamma)^n}{1-q_\gamma^n}.
\]

\begin{theorem}[Millson]
The central value of the Selberg zeta function of odd type on a $3$-dimensional hyperbolic manifold 
$M=\Gamma\backslash \bH^3$ is related to the eta invariant by the identity
\[
\exp(i\pi\eta(M))=\cZ_\Gamma(0).
\]
 
\end{theorem}

\end{document}